\documentclass[11pt]{article}

\usepackage[a4paper,margin=1in]{geometry}
\usepackage{amsmath,amssymb,amsthm,mathtools}
\usepackage{bm}
\usepackage{mathrsfs}
\usepackage{threeparttable}
\usepackage{enumitem}
\usepackage[expansion=false]{microtype}
\usepackage{xcolor}
\usepackage[colorlinks=true,linkcolor=blue,citecolor=blue,urlcolor=blue]{hyperref}

\numberwithin{equation}{section}

\newtheorem{theorem}{Theorem}[section]
\newtheorem{proposition}[theorem]{Proposition}
\newtheorem{lemma}[theorem]{Lemma}

\theoremstyle{definition}
\newtheorem{definition}[theorem]{Definition}
\newtheorem{assumption}[theorem]{Assumption}
\theoremstyle{remark}
\newtheorem{remark}[theorem]{Remark}

\newcommand{\R}{\mathbb R}
\newcommand{\E}{\mathbb E}
\newcommand{\Pp}{\mathbb P}
\newcommand{\dd}{\,\mathrm d}
\newcommand{\K}{\mathbb K}
\newcommand{\norm}[1]{\left\lVert#1\right\rVert}

\newcommand{\ip}[2]{\left\langle#1,#2\right\rangle}

\def\sF{{\mathscr F}}

\def\bP{{\mathbf P}}

\title{Small time asymptotics for kinetic SDEs with singular drift}
\author{
Chongyang Ren\textsuperscript{1},
Saisai Yang\textsuperscript{1},
Tusheng Zhang\textsuperscript{1,2}
}
\date{}

\begin{document}
\maketitle
\footnotetext[1]{%
School of Mathematics, University of Science and Technology of China,
Hefei, China.
Email: \texttt{rcy.math@ustc.edu.cn} (Chongyang Ren), \texttt{yangss@ustc.edu.cn} (Saisai Yang).
}

\footnotetext[2]{%
Department of Mathematics, University of Manchester,
Manchester M13 9PL, United Kingdom.
Email: \texttt{tusheng.zhang@manchester.ac.uk}.
}

\begin{abstract}
We study small-time asymptotics for kinetic stochastic differential equations with singular drifts in kinetic Kato classes. We establish a small-time large deviation principle and prove a pointwise Varadhan-type asymptotic for the associated transition density.
\end{abstract}

\section{Introduction and Framework}

In this paper, we study the small-time behavior of the transition density associated with the kinetic stochastic differential equation
\begin{equation}\label{eq:sde}
  \dd X_t=V_t\dd t,\qquad
  \dd V_t=b(t,X_t,V_t)\dd t+\sqrt{2}\dd W_t,
\end{equation}
where $(W_t)_{t\ge0}$ is a $d$-dimensional Brownian motion on the stochastic basis
$
 \mathfrak F=(\Omega,\sF,(\sF_t)_{t\ge0},\bP),
$
and $b:\R_+\times\R^{2d}\to\R^d$ is Borel measurable. 

Kinetic stochastic differential equations of the form \eqref{eq:sde},
often viewed as stochastic Hamiltonian or Langevin systems, arise naturally
in kinetic theory, nonequilibrium statistical mechanics, and stochastic
models of mechanical systems; see, for instance,
\cite{vanKampen1981,Villani2009} and the references therein.
A characteristic feature of these equations is their hypoelliptic structure:
the noise acts only on the velocity variable and is transmitted to the
position variable through the transport term. This mechanism goes back to
the classical works of Kolmogorov \cite{Kol34} and H\"ormander
\cite{Hormander1967}. 
Questions of well-posedness and heat-kernel estimates for kinetic stochastic differential equations with irregular coefficients have been studied extensively; see, among others,
\cite{Chaudru2017,CdRMPZ2023,RenZhang2025}. Such estimates provide sharp bounds on the transition density, but they generally do not reveal its precise exponential decay rate in the small-time regime. The main objective of this work is to characterize the small-time logarithmic asymptotics of the transition density.

The classical small-time behavior is already illustrated by Brownian motion.
Let $B_t=\sqrt{2}\,W_t$ start from the origin, and denote its density by
$g_t$. Then

$$
g_t(y)
=
(4\pi t)^{-d/2}
\exp\left(-\frac{|y|^2}{4t}\right),
\qquad y\in\mathbb R^d,
$$
and for every fixed $y\in\mathbb R^d$,
$$
\lim_{t\downarrow0} t\log g_t(y)
=
-\frac{|y|^2}{4}.
$$
Varadhan \cite{Varadhan1967a,Varadhan1967b} extended this logarithmic asymptotic to uniformly elliptic diffusions. By contrast, the kinetic setting exhibits an intrinsically anisotropic short-time structure. This is visible in the drift-free case: when $b=0$, \eqref{eq:sde} reduces to the Kolmogorov process

$$
\dd X_t=V_t\dd t,\qquad
\dd V_t=\sqrt2\,\dd W_t .
$$
Its transition density is explicit; see \cite{Kol34}:
\begin{equation}\label{eq:kolmogorov-kernel}
 p_0(s,\bm z;t,\bm y)
 =\left(\frac{\sqrt3}{2\pi (t-s)^2}\right)^d
 \exp\left[
 -\frac12
 \left\langle
 Q_{t-s}^{-1}\bigl(\bm y-\theta_{s,t}(\bm z)\bigr),
 \bm y-\theta_{s,t}(\bm z)
 \right\rangle
 \right],
\end{equation}
where
\begin{align}\label{eq:theta}
 \theta_{s,t}(x,v)=\bigl(x+(t-s)v,v\bigr),\qquad
 Q_t=
 \begin{pmatrix}
   \frac{2t^3}{3}I_d&t^2I_d\\
   t^2I_d&2tI_d
 \end{pmatrix}.
\end{align}
In particular, starting from the origin,
\begin{equation}\label{eq:kolmogorov-origin-density}
 p_t(x,v)
 =
 \left(\frac{\sqrt3}{2\pi t^2}\right)^d
 \exp\left(
 -\frac{3|2x-tv|^2+t^2|v|^2}{4t^3}
 \right).
\end{equation}
A direct calculation from \eqref{eq:kolmogorov-origin-density} gives, for
fixed $(x,v)\in\R^{2d}$,
\begin{equation}\label{eq:kolmogorov-fixed-endpoint-limits}
 \lim_{t\downarrow0}t\log p_t(x,v)
 =
 \begin{cases}
   -\infty, & x\ne0,\\
   -|v|^2, & x=0,
 \end{cases}
 \qquad
 \lim_{t\downarrow0}t^3\log p_t(x,v)
 =
 -3|x|^2.
\end{equation}
When $x=0$, the first limit in \eqref{eq:kolmogorov-fixed-endpoint-limits} recovers the classical diffusive behavior at speed $t^{-1}$, in agreement with the uniformly elliptic Varadhan asymptotic; see \cite{Varadhan1967a,Varadhan1967b}. The situation changes qualitatively when $x\neq0$: the quantity $t\log p_t(x,v)$ diverges to $-\infty$, and hence the $t^{-1}$ scale no longer yields a finite logarithmic asymptotic. At the faster scale $t^{-3}$, the limit becomes finite but is independent of the velocity variable. Thus the asymptotics exhibit two distinct scales: the diffusive scale $t^{-1}$ in the pure velocity direction and the faster scale $t^{-3}$ in the presence of a spatial displacement. Moreover, the latter scale captures only the leading spatial contribution. This suggests that, in order to retain nontrivial contributions from both components, one should rescale the endpoint anisotropically.

These asymptotics reflect the different short-time scales of the two components. Indeed, for the Kolmogorov process starting from the origin,
$$
X_t=\sqrt{2}\int_0^t W_r\,\dd r,
\qquad
V_t=\sqrt{2}\,W_t,
$$
so that the typical fluctuations of $V_t$ and $X_t$ are of orders $t^{1/2}$ and $t^{3/2}$, respectively. This difference also determines the appropriate endpoint scaling at the large-deviation speed $t^{-1}$. To retain nontrivial contributions from both components, the spatial displacement must shrink with time. More precisely, setting $x=tq_1$ and $v=q_2$ in \eqref{eq:kolmogorov-origin-density}, one obtains

$$
\lim_{t\downarrow0}
t\log p_t(D_tq)
=
\lim_{t\downarrow0}
t\log p_t(tq_1,q_2)
=
-\left(
3|q_1|^2
-3\langle q_1,q_2\rangle
+|q_2|^2
\right).
$$
where 
\begin{align}\label{eq:D}
D_t=
\begin{pmatrix}
tI_d&0\\
0&I_d
\end{pmatrix}.
\end{align}
Thus $D_t$ provides the natural endpoint scaling for which the position and velocity variables contribute simultaneously to the limiting action.

The preceding computation leads naturally to the question of whether the same small-time asymptotics persist in the presence of a singular drift. In the non-degenerate setting, Chen, Fang and Zhang \cite{ChenFangZhang2019} established a small-time large deviation principle and a pointwise Varadhan-type asymptotic for Brownian motion with measure-valued drift in a Kato class, showing that the singular drift does not affect the leading small-time action. The kinetic setting is more delicate, since the degeneracy of the dynamics already produces different small-time scales in the drift-free case. The purpose of the present paper is to establish the corresponding small-time large deviation principle and pointwise Varadhan-type asymptotics for kinetic stochastic differential equations with singular drifts in the kinetic Kato classes.

We now introduce the kinetic Kato classes defined in \cite{RenZhang2025}.  For $\beta\ge0$ and $\lambda>0$, set
$$
\eta_{\lambda}^{(\beta)}(t,x,v)
 :=
 t^{-\frac{\beta}{2}-2d}
 \exp\bigl(-\lambda\rho_t(x,v)\bigr),
 \qquad t>0,
$$
where
$$
\rho_t(x,v)
 :=
 \frac{|x-tv|^2}{t^3}+\frac{|v|^2}{t}.
$$
The explicit Kolmogorov kernel \eqref{eq:kolmogorov-kernel} implies that there exist constants $c_0,\lambda_0,c_1,\lambda_1>0$ such that
$$
c_0\,\eta_{\lambda_0}^{(0)}(t,x,v)
 \le p_t(x,v)
 \le c_1\,\eta_{\lambda_1}^{(0)}(t,x,v),
 \qquad t>0.
$$

Let $f:\R_+\times\R^{2d}\to\R$ be measurable, extended by zero to negative times.  For $\delta>0$, define
$$
\begin{aligned}
 K_{\lambda}^{(\beta)}(f;\delta)
 &:=
 \sup_{t\ge0}\sup_{(x,v)\in\R^{2d}}
 \int_0^\delta\int_{\R^{2d}}
 \eta_{\lambda}^{(\beta)}(r,y,w) \\
 &\qquad\qquad\times
 \bigl|f(t\pm r,x\pm y,v\pm w)\bigr|
 \,\mathrm dy\,\mathrm dw\,\mathrm dr ,
\end{aligned}
$$
where
$$
\bigl|f(t\pm r,x\pm y,v\pm w)\bigr|
 :=
 \sum_{i,j,k=0}^{1}
 \bigl|f\bigl(t+(-1)^i r,x+(-1)^j y,v+(-1)^k w\bigr)\bigr|.
$$

\begin{definition}[Kinetic Kato class]
For $\beta\ge0$, define
$$
 \K_\beta
 :=
 \left\{
 f\in L_{\mathrm{loc}}^1(\R_+\times\R^{2d})
 :
 K_{\lambda}^{(\beta)}(f;1)<\infty,\quad
 \lim_{\delta\downarrow0}K_{\lambda}^{(\beta)}(f;\delta)=0,
 \ \forall\lambda>0
 \right\}.
$$
\end{definition}

These classes contain the mixed Lebesgue regimes naturally adapted to the kinetic scaling.  For $p_x,p_v\in[1,\infty]$, write
$$
 \|f\|_{L_v^{p_v}L_x^{p_x}}
 :=
 \left(
 \int_{\R^d}
 \left(
 \int_{\R^d}|f(x,v)|^{p_x}\,\mathrm dx
 \right)^{p_v/p_x}\mathrm dv
 \right)^{1/p_v},
$$
with the usual modifications when an exponent is infinite.  By H\"older's inequality and \cite[Lemma~2.6]{RenZhang2025},
$$
K_{\lambda}^{(\beta)}(f;\delta)
 \lesssim_C
 \delta^{\frac{2-\beta-\kappa}{2}}
 \|f\|_{L^q(\R_+;L_v^{p_v}L_x^{p_x})},
 \qquad
 \kappa:=\frac2q+\frac{3d}{p_x}+\frac d{p_v}<2-\beta,
$$
for a constant $C=C(d,p_x,p_v,q,\beta,\lambda)$.  Consequently,
$$
 L^q(\R_+;L_v^{p_v}L_x^{p_x})\subset\K_\beta.
$$
Notice that the two Kato conditions appearing below are not equivalent.  For instance, near the origin
$$
 f(t)=t^{-1/2}|\log t|^{-1},\qquad 0<t<\frac12,
$$
satisfies $|f|^2\in\K_0$ but $f\notin\K_1$.

\begin{definition}\label{def:weak-solution}
Let $\mathfrak{F}:=(\Omega, \mathscr{F}, (\mathscr{F}_t)_{t\ge0}, \Pp)$ be a stochastic basis.
We call  a pair of $\sF_t$-adapted processes $(Z,W)$ on $\mathfrak{F}$, with $Z=(X,V)$, a weak solution of SDE \eqref{eq:sde} starting from point initial point $ z=(x,v)$ at time $s$ if 
\begin{enumerate}
\item[(i)] $\bP(Z_s=z)=1$; (ii) $W$ is a $d$-dimensional $\mathscr{F}_t$-Brownian motion;
\item[(iii)]  For any $t>0$, it holds that $\int_s^t|b(r,Z_r)|\dd r<\infty$ a.s. and
\begin{align*}
X_t=x+\int_s^tV_r\dd r,\qquad
 V_t=v+\int_s^tb(r,Z_r)\dd r+\sqrt2\,(W_t-W_s),\ \ a.s.
\end{align*}
\end{enumerate}
\end{definition}

\begin{assumption}\label{ass:drift}
For the path-space large-deviation result, assume
\begin{equation}\label{eq:assumptions}
 |b|\in\K_1,\qquad |b|^2\in\K_0.
\end{equation}
For the heat-kernel estimate and the pointwise small-time asymptotics, assume in addition that
\begin{equation}\label{eq:heat-assumption}
 |b|^\gamma\in\K_1
 \qquad\text{for some }\gamma>1.
\end{equation}
\end{assumption}

The assumptions above play different roles.  The condition $|b|\in\K_1$ gives weak well-posedness, while $|b|^2\in\K_0$ yields the short-time exponential integrability needed to prove exponential equivalence.  The additional condition \eqref{eq:heat-assumption} is used through the two-sided heat-kernel estimate.

We shall repeatedly use the following consequence of \cite[Theorem~1.4]{RenZhang2025}.
\begin{proposition}\label{prop:weak-wellposedness}
Assume that $|b|\in\K_1$.  Then equation~\eqref{eq:sde} admits a unique weak solution for every deterministic initial point.  If, in addition, \eqref{eq:heat-assumption} holds, then this solution admits a transition density $p_b(s,\bm u;t,\bm y)$.  Moreover, for every $T>0$, there exist constants $c_0,c_1,C_0,C_1>0$ such that, for all $0<t-s\le T$,
\begin{equation}\label{eq:heat-kernel-bounds}
\begin{aligned}
 C_0(t-s)^{-2d}
 \exp\!\left(
 -c_0\rho_{t-s}\bigl(\bm y-\theta_{s,t}(\bm u)\bigr)
 \right)
 &\le
 p_b(s,\bm u;t,\bm y) \\
 &\le
 C_1(t-s)^{-2d}
 \exp\!\left(
 -c_1\rho_{t-s}\bigl(\bm y-\theta_{s,t}(\bm u)\bigr)
 \right).
\end{aligned}
\end{equation}
\end{proposition}

The rest of the paper is organized as follows.  Section~2 states the main results and records the endpoint action associated with the kinetic large-deviation principle.  Section~3 proves the path-space LDP through Krylov--Khasminskii estimates and exponential equivalence.  Section~4 proves the pointwise small-time asymptotics by combining the LDP, the heat-kernel bounds, and the Chapman--Kolmogorov argument.

\section{Main results}

Fix $s\ge0$ and an initial point $\bm z=(x,v)\in\mathbb R^{2d}$, and let $(X_t,V_t)_{t\ge s}$ denote the solution of \eqref{eq:sde} starting from
$\bm z$ at time $s$. For $\varepsilon>0$, we consider the anisotropically rescaled process:
\begin{equation}\label{eq:scaled-process}
Y_t^\varepsilon
:=D_\varepsilon^{-1}
\left(
(X_{s+\varepsilon t},V_{s+\varepsilon t})
-\theta_{s,s+\varepsilon t}(\bm z)
\right),
\qquad 0\le t\le1.
\end{equation}
Using the definitions of $\theta_{s,s+\varepsilon t}$ in \eqref{eq:theta} and $D_\varepsilon$ in \eqref{eq:D}, we obtain
$$
Y_t^{\varepsilon,1}
=
\frac{X_{s+\varepsilon t}-x-\varepsilon t v}{\varepsilon},
\qquad
Y_t^{\varepsilon,2}
=
V_{s+\varepsilon t}-v.
$$
For $r\in(0,1]$, $\phi=(\phi_1, \phi_2)\in C([0,r];\R^{2d})$, define the rate function $I_r$ by
\begin{equation}\label{eq:time-r-action}
 I_r(\phi)=
 \begin{cases}
 \displaystyle\frac14\int_0^r|\dot\phi_2(\tau)|^2\dd\tau,
 &\begin{array}{l}
   \phi_1\text{ is absolutely continuous},\quad
   \phi_2\in H^1([0,r];\R^d),\\
   \phi(0)=0,\quad
   \dot\phi_1(\tau)=\phi_2(\tau)\quad\text{for a.e. }\tau\in[0,r],
  \end{array}\\[5mm]
 +\infty,&\text{otherwise}.
 \end{cases}
\end{equation}

\begin{theorem}\label{thm:path-ldp}
Under the assumption \eqref{eq:assumptions}, for every fixed $s\ge0$ and
$\bm z\in\R^{2d}$, the laws of $Y^\varepsilon$ under $\Pp_{s,\bm z}$ on
$C([0,1];\R^{2d})$ satisfy an LDP with speed $1/\varepsilon$ and good
rate function $I_1$ defined by \eqref{eq:time-r-action}.  More precisely,
let
$$
 \mathcal X:=C([0,1];\R^{2d}),
 \qquad
 \nu_{s,\bm z}^\varepsilon
 :=\mathcal L_{s,\bm z}(Y^\varepsilon),
$$
where $\mathcal X$ is equipped with the uniform topology, we have
\begin{enumerate}[label=\textup{(\roman*)}]
 \item for every closed subset $C\subset\mathcal X$,
 $$
\limsup_{\varepsilon\downarrow0}
  \varepsilon\log\nu_{s,\bm z}^\varepsilon(C)
  \le -\inf_{\phi\in C_0}I_1(\phi),
$$
 where $C_0:=\{\phi\in C:\phi(0)=0\}$;

 \item for every open subset $G\subset\mathcal X$,
 $$
\liminf_{\varepsilon\downarrow0}
  \varepsilon\log\nu_{s,\bm z}^\varepsilon(G)
  \ge -\inf_{\phi\in G_0}I_1(\phi),
$$
 where $G_0:=\{\phi\in G:\phi(0)=0\}$.
\end{enumerate}
\end{theorem}

\begin{theorem}\label{thm:varadhan}
Under \eqref{eq:assumptions} and \eqref{eq:heat-assumption}, for every fixed
$s\ge0$, $\bm z\in\R^{2d}$,
and $q=(q_1,q_2)\in\R^{2d}$,
$$
\lim_{\varepsilon\downarrow0}\varepsilon\log
 p_b\!\left(s,\bm z;s+\varepsilon,
 \theta_{s,s+\varepsilon}(\bm z)+D_\varepsilon q\right)
 =-J_1(q),
$$
where
$ J_1(q)=3|q_1|^2-3\ip{q_1}{q_2}+|q_2|^2$.
\end{theorem}

\begin{theorem}\label{thm:fixed-endpoint}
Assume \eqref{eq:assumptions} and \eqref{eq:heat-assumption}.  Let $s\ge0$
and $\bm z\in\R^{2d}$ be fixed.
\begin{enumerate}[label=\textup{(\roman*)}]
 \item For every fixed $q=(q_1,q_2)\in\R^{2d}$,
 \begin{equation}\label{eq:fixed-endpoint}
  \lim_{\varepsilon\downarrow0}\varepsilon^3\log
  p_b\!\left(s,\bm z;s+\varepsilon,
  \theta_{s,s+\varepsilon}(\bm z)+q\right)
  =-3|q_1|^2 .
 \end{equation}
 In particular, the limit is independent of the velocity displacement $q_2$.
 \item If $q_1=0$, the left-hand side of \eqref{eq:fixed-endpoint} vanishes,
 and the exact decay is captured at the slower speed $\varepsilon^{-1}$:
 \begin{equation}\label{eq:pure-velocity-endpoint}
  \lim_{\varepsilon\downarrow0}\varepsilon\log
  p_b\!\left(s,\bm z;s+\varepsilon,
  \theta_{s,s+\varepsilon}(\bm z)+(0,q_2)\right)
  =-|q_2|^2 .
 \end{equation}
\end{enumerate}
\end{theorem}
\begin{remark}
Theorem~\ref{thm:fixed-endpoint} reveals a hierarchy of exponential scales inherent in the kinetic dynamics.  For a fixed endpoint displacement $q=(q_1,q_2)$ with $q_1\neq0$, the leading small-time asymptotic occurs at speed $\varepsilon^{-3}$ and depends only on the spatial component $q_1$.  The velocity displacement $q_2$ is therefore invisible at this scale.  When $q_1=0$, the $\varepsilon^{-3}$ contribution vanishes, and the next nontrivial asymptotic appears at the slower speed $\varepsilon^{-1}$, yielding the velocity cost $|q_2|^2$.

This hierarchy reflects the degenerate nature of the kinetic system: the velocity component is directly forced by the Brownian noise, whereas the spatial component is obtained by integrating the velocity.  Accordingly, fixed spatial and velocity displacements contribute at different exponential orders.  In particular, part~(ii) identifies the next-order asymptotic precisely in the regime where the leading spatial cost in part~(i) vanishes.
\end{remark}

\section{Proof of Theorem~\ref{thm:path-ldp}}
\label{sec:proof-path-ldp}

The following conditional kinetic Krylov estimate is
\cite[Lemma~2.11]{HZZ2026} adapted in the present setting.

\begin{lemma}
Let $Z=(X,V)$ be a weak solution of \eqref{eq:sde} under
\eqref{eq:assumptions}.  For every $T>0$, there exist constants
$\lambda_*\in(0,1)$ and $C_{T,b}>0$ such that, for every non-negative
$g\in\K_0$ and $0\le r\le t\le T$,
\begin{equation}\label{eq:krylov-input}
 \E\left[\left.\int_r^t g(\tau,Z_\tau)\dd\tau
 \right|\mathcal F_r\right]
 \le C_{T,b}K_{\lambda_*}^{(0)}(g;t-r)
 \qquad\text{a.s.}
\end{equation}
\end{lemma}

\begin{proposition}
\label{prop:khasminskii}
Let $Z=(X,V)$ be a weak solution of \eqref{eq:sde} under
\eqref{eq:assumptions}, and let $f\in\K_0$ be non-negative.  For every
$T>0$, $a>0$, and $\delta\in(0,1)$, there exists
$\delta_0=\delta_0(a,\delta,T,b,f)>0$ such that for all $0\le s\le t\le T$ with $t-s\le \delta_0$
\begin{equation}\label{eq:khasminskii-unconditional}
 \E_{s,\bm z}\exp\left(a\int_s^t f(r,Z_r)\dd r\right)
 \le\frac1{1-\delta}
\qquad\text{uniformly in }\bm z\in\R^{2d}.
\end{equation}
\end{proposition}

\begin{proof}
Set
$$
 M_f(h):=C_{T,b}K_{\lambda_*}^{(0)}(f;h),
 \qquad 0\le h\le T,
$$
where $C_{T,b}$ is the constant appeared in \eqref{eq:krylov-input}.
The map $h\mapsto M_f(h)$ is non-decreasing, and
\eqref{eq:krylov-input} reads
\begin{equation}\label{eq:khasminskii-first-moment}
 \E\left[\left.\int_r^t f(u,Z_u)\dd u\right|\mathcal F_r\right]
 \le M_f(t-r),
 \qquad 0\le r\le t\le T.
\end{equation}
We claim that, for every $n\in\mathbb N$,
\begin{equation}\label{eq:khasminskii-conditional-moments}
 \E\left[\left.
  \left(\int_r^t f(u,Z_u)\dd u\right)^n
  \right|\mathcal F_r\right]
 \le n!\,M_f(t-r)^n
 \qquad\text{a.s.}
\end{equation}
This is the standard Khasminskii iteration used in the proof of
\cite[Lemma~2.11]{HZZ2026}; we include the details for completeness.  The case $n=1$ is
\eqref{eq:khasminskii-first-moment}.  Suppose that
\eqref{eq:khasminskii-conditional-moments} holds with $n-1$ in place
of $n$.  Elementary calculus  gives
$$
 \left(\int_r^t f(u,Z_u)\dd u\right)^n
 =n\int_r^t f(u,Z_u)
       \left(\int_u^t f(v,Z_v)\dd v\right)^{n-1}\dd u.
$$
Conditioning first on $\mathcal F_u$ and then on $\mathcal F_r$, and
using the induction hypothesis, we obtain
\begin{align*}
 &\E\left[\left.
   \left(\int_r^t f(u,Z_u)\dd u\right)^n
   \right|\mathcal F_r\right]                                      \\
 &\quad=n\,\E\left[\left.
   \int_r^t f(u,Z_u)
   \E\left[\left.
    \left(\int_u^t f(v,Z_v)\dd v\right)^{n-1}
    \right|\mathcal F_u\right]\dd u
   \right|\mathcal F_r\right]                                     \\
 &\quad\le n!\,M_f(t-r)^{n-1}
   \E\left[\left.\int_r^t f(u,Z_u)\dd u
   \right|\mathcal F_r\right]                                     \\
 &\quad\le n!\,M_f(t-r)^n,
\end{align*}
where we used $M_f(t-u)\le M_f(t-r)$ in the first inequality and
\eqref{eq:khasminskii-first-moment} in the last.  This proves the claim.

Because $f\in\K_0$, we have $M_f(h)\to0$ as $h\downarrow0$.  Choose
$\delta_0\in(0,T]$ so that
\begin{equation}\label{eq:khasminskii-smallness}
 aM_f(\delta_0)\le\delta.
\end{equation}
For $0\le s\le t\le T$ with $t-s\le\delta_0$, conditional monotone
convergence, \eqref{eq:khasminskii-conditional-moments}, and
\eqref{eq:khasminskii-smallness} yield
\begin{align*}
 \E\left[\left.
  \exp\left(a\int_s^t f(r,Z_r)\dd r\right)
  \right|\mathcal F_s\right]
 &=\sum_{n=0}^\infty\frac{a^n}{n!}
   \E\left[\left.
    \left(\int_s^t f(r,Z_r)\dd r\right)^n
    \right|\mathcal F_s\right]                                    \\
 &\le\sum_{n=0}^\infty\bigl(aM_f(t-s)\bigr)^n
 \le\frac1{1-\delta}
 \qquad\text{a.s.}
\end{align*}
Thus the stronger conditional estimate holds.  Taking expectation under
$\Pp_{s,\bm z}$ and then the supremum over $(s,t,\bm z)$ gives
\eqref{eq:khasminskii-unconditional}.
\end{proof}
Now we give the proof of Theorem~\ref{thm:path-ldp}.
\begin{proof}[Proof of Theorem~\ref{thm:path-ldp}]
We first show that the drift contribution is negligible at a
superexponential rate. More precisely, for every $\rho>0$,
\begin{equation}\label{eq:superexp}
 \limsup_{\varepsilon\downarrow0}
 \varepsilon\log
 \sup_{\bm z\in\R^{2d}}
 \Pp_{s,\bm z}\left(
  \int_s^{s+\varepsilon}
  |b(r,X_r,V_r)|\,\dd r>\rho
 \right)
 =-\infty.
\end{equation}

Fix $s\ge0$ and choose $T>s+1$. Let $a>0$ be arbitrary. By
Proposition~\ref{prop:khasminskii}, applied with $f=|b|^2$, there
exists $\delta_a>0$ such that for $0<\varepsilon\le \delta_a\wedge1$,
\begin{equation}\label{eq:khasminskii-b-small}
 \sup_{\bm z\in\R^{2d}}
 \E_{s,\bm z}\exp\left(
  a\int_s^{s+\varepsilon}
  |b(r,X_r,V_r)|^2\,\dd r
 \right)
 \le 2.
\end{equation}
By the Cauchy--Schwarz inequality,
$$
 \left(
  \int_s^{s+\varepsilon}
  |b(r,X_r,V_r)|\,\dd r
 \right)^2
 \le
 \varepsilon
 \int_s^{s+\varepsilon}
 |b(r,X_r,V_r)|^2\,\dd r.
$$
By Chebyshev's inequality and
\eqref{eq:khasminskii-b-small}, we have
\begin{align*}
 &\sup_{\bm z\in\R^{2d}}
 \Pp_{s,\bm z}\left(
  \int_s^{s+\varepsilon}
  |b(r,X_r,V_r)|\,\dd r>\rho
 \right)\\
 &\quad\le
 \sup_{\bm z\in\R^{2d}}
 \Pp_{s,\bm z}\left(
  \int_s^{s+\varepsilon}
  |b(r,X_r,V_r)|^2\,\dd r
  >\frac{\rho^2}{\varepsilon}
 \right)\\
 &\quad\le
 \exp\left(-\frac{a\rho^2}{\varepsilon}\right)
 \sup_{\bm z\in\R^{2d}}
 \E_{s,\bm z}\exp\left(
  a\int_s^{s+\varepsilon}
  |b(r,X_r,V_r)|^2\,\dd r
 \right)\\
 &\quad\le
 2\exp\left(-\frac{a\rho^2}{\varepsilon}\right).
\end{align*}
It follows that
$$
 \limsup_{\varepsilon\downarrow0}
 \varepsilon\log
 \sup_{\bm z\in\R^{2d}}
 \Pp_{s,\bm z}\left(
  \int_s^{s+\varepsilon}
  |b(r,X_r,V_r)|\,\dd r>\rho
 \right)
 \le -a\rho^2.
$$
Since $a>0$ was arbitrary, \eqref{eq:superexp} follows.

We now compare the solution with the free rescaled process.  Define
$$
 G_t^\varepsilon
 :=\left(\sqrt2\int_0^t
       (W_{s+\varepsilon r}-W_s)\dd r,
       \sqrt2(W_{s+\varepsilon t}-W_s)\right).
$$
From \eqref{eq:sde} and \eqref{eq:scaled-process},
$$
 Y_t^{\varepsilon,1}=\int_0^tY_r^{\varepsilon,2}\dd r,
\qquad
 Y_t^{\varepsilon,2}
 =\sqrt2(W_{s+\varepsilon t}-W_s)
 +\varepsilon\int_0^t b(s+\varepsilon r,X_{s+\varepsilon r},V_{s+\varepsilon r})\dd r.
$$
Therefore
\begin{equation}\label{eq:scaled-comparison}
 \norm{Y^\varepsilon-G^\varepsilon}_\infty
 \le 2\int_s^{s+\varepsilon}|b(r,X_r,V_r)|\dd r.
\end{equation}
The estimate \eqref{eq:superexp} gives, for every $\rho>0$,
$$
 \limsup_{\varepsilon\downarrow0}\varepsilon\log
 \Pp_{s,\bm z}\left(\norm{Y^\varepsilon-G^\varepsilon}_\infty>\rho\right)=-\infty.
$$
Thus
$Y^\varepsilon$ and $G^\varepsilon$ are exponentially equivalent at speed
$1/\varepsilon$ in the sense of
\cite[Theorem~4.2.13]{DemboZeitouni1998}.

By Brownian scaling,
$$
 G^\varepsilon
 \stackrel{\mathrm{law}}{=}
 \sqrt{\varepsilon}
 \left(
  \sqrt2\int_0^\cdot B_\tau\,\dd\tau,
  \sqrt2 B
 \right),
$$
where $B$ is a standard Brownian motion. Since the mapping
$$
f\longmapsto
 \left(
  \sqrt2\int_0^\cdot f(\tau)\,\dd\tau,
  \sqrt2 f
 \right)
$$
is continuous from $C_0([0,1];\R^d)$ into
$C([0,1];\R^{2d})$, Schilder's theorem and the contraction
principle imply that $\{G^\varepsilon\}_{\varepsilon>0}$
satisfies an LDP with speed $\varepsilon^{-1}$ and good rate
function $I_1$ defined in \eqref{eq:time-r-action}. Finally,
the exponential equivalence of $Y^\varepsilon$ and
$G^\varepsilon$ yields the same LDP for
$\{Y^\varepsilon\}_{\varepsilon>0}$, completing the proof.
\end{proof}

\section{Proofs of the small-time asymptotics}
\label{sec:small-time}

\begin{proof}[Proof of Theorem~\ref{thm:varadhan}]
By \eqref{eq:time-r-action} and Fubini's theorem, we have
\begin{align*}\phi_1(r)
 =
 \int_0^r\int_0^t\dot\phi_2(\tau)\,\dd\tau\,\dd t
 =
 \int_0^r(r-\tau)\dot\phi_2(\tau)\,\dd\tau,
\quad
 \phi_2(r)=\int_0^r\dot\phi_2(\tau)\,\dd\tau.
\end{align*}
By the contraction principle, for fixed $r\geq 0$, the family
$\{Y_r^\varepsilon\}_{\varepsilon>0}$ satisfies an LDP on
$\R^{2d}$ with speed $\varepsilon^{-1}$ and good rate function
$$
J_r(q)
 :=
 \inf\left\{
 I_r(\phi):
 \phi\in C([0,r];\R^{2d}),\ \phi(0)=0,\ \phi(r)=q
 \right\}, \quad q=(q_1,q_2)\in\R^{2d}.
$$
By the standard minimum-energy formula for controllable linear systems
(see, e.g., \cite[Section~11.5]{Hespanha2018}), the endpoint rate function is given by 
\begin{equation}\label{eq:endpoint-action-computed}
 J_r(q)
 =
 \frac12 q^\top Q_r^{-1}q
 =
 \frac{3}{r^3}|q_1|^2
 -\frac{3}{r^2}\ip{q_1}{q_2}
 +\frac1r|q_2|^2,
\end{equation}
where
$$
 Q_r^{-1}
 =
 \begin{pmatrix}
  \dfrac6{r^3}I_d&-\dfrac3{r^2}I_d\\[2mm]
  -\dfrac3{r^2}I_d&\dfrac2r I_d
 \end{pmatrix}.
$$
In particular,
$$
J_1(q)
 =
 3|q_1|^2
 -3\ip{q_1}{q_2}
 +|q_2|^2.
$$

Fix $\eta\in(0,1)$ and put
$$
 t_\eta=s+(1-\eta)\varepsilon,
 \qquad
 \qquad
 \bm y_\varepsilon
 =\theta_{s,s+\varepsilon}(\bm z)+D_\varepsilon q.
$$
The Chapman--Kolmogorov identity gives
\begin{align}
 p_b(s,\bm z;s+\varepsilon,\bm y_\varepsilon)
 &=\int_{\R^{2d}}
   p_b(s,\bm z;t_\eta,\bm u)
   p_b(t_\eta,\bm u;s+\varepsilon,\bm y_\varepsilon)\dd\bm u \notag\\
 &=\E_{s,\bm z}\!\left[
   p_b\!\left(t_\eta,
   \theta_{s,t_\eta}(\bm z)
      +D_\varepsilon Y_{1-\eta}^\varepsilon;
   s+\varepsilon,\bm y_\varepsilon\right)
   \right],                                            \label{eq:ck-density-transfer}
\end{align}
where the last step is due to the fact that
$(X_{t_\eta},V_{t_\eta})=\theta_{s,t_\eta}(\bm z)+D_\varepsilon Y_{1-\eta}^\varepsilon$.

For $\zeta=(\zeta_1,\zeta_2)\in\R^{2d}$, define
\begin{equation}\label{eq:terminal-cost}
 \Psi_\eta(q,\zeta)
 :=\frac{|q_1-\zeta_1-\eta q_2|^2}{\eta^3}
   +\frac{|q_2-\zeta_2|^2}{\eta}.
\end{equation}
By the definition of $\rho_{\eta\varepsilon}$ in Section 1, we have
\begin{align}\label{eq:last-segment}
 &\rho_{\eta\varepsilon}\!\left(
 \bm y_\varepsilon-
 \theta_{t_\eta,s+\varepsilon}\!\left(
   \theta_{s,t_\eta}(\bm z)+D_\varepsilon\zeta
 \right)\right)\nonumber\\                                                   
 &=\frac{\left|
  \varepsilon(q_1-\zeta_1-\eta\zeta_2)
  -\eta\varepsilon(q_2-\zeta_2)\right|^2}
  {(\eta\varepsilon)^3}
  +\frac{|q_2-\zeta_2|^2}{\eta\varepsilon}                      
 =\frac1\varepsilon\Psi_\eta(q,\zeta).
\end{align}
Combining \eqref{eq:ck-density-transfer}, \eqref{eq:last-segment}, and
\eqref{eq:heat-kernel-bounds}, we obtain
\begin{align}
 &C_0(\eta\varepsilon)^{-2d}
 \E_{s,\bm z}\exp\!\left(
   -\frac{c_0}{\varepsilon}
    \Psi_\eta(q,Y_{1-\eta}^\varepsilon)
 \right)
 \le p_b(s,\bm z;s+\varepsilon,\bm y_\varepsilon) \notag\\
 &\hspace{18mm}\le
 C_1(\eta\varepsilon)^{-2d}
 \E_{s,\bm z}\exp\!\left(
   -\frac{c_1}{\varepsilon}
    \Psi_\eta(q,Y_{1-\eta}^\varepsilon)
 \right).                                               \label{eq:density-laplace-bounds}
\end{align}
For fixed $\eta\in(0,1)$ and $c>0$, define
$$
 F(\zeta):=-c\Psi_\eta(q,\zeta),
 \qquad \zeta\in\R^{2d}.
$$
The continuity and nonnegativity of $\Psi_\eta(q,\cdot)$ imply that $F$ is continuous and bounded above by zero. Since $Y_{1-\eta}^{\varepsilon}$ satisfies a good LDP on $\R^{2d}$ with speed
$\varepsilon^{-1}$ and rate function $J_{1-\eta}$, Varadhan's integral
lemma \cite[Theorem~4.3.1]{DemboZeitouni1998} yields
\begin{align}
 &\lim_{\varepsilon\downarrow0}
 \varepsilon\log
 \E_{s,\bm z}\exp\!\left(
  -\frac{c}{\varepsilon}
   \Psi_\eta(q,Y_{1-\eta}^{\varepsilon})
 \right) \notag\\
 &\qquad=
 \sup_{\zeta\in\R^{2d}}
 \bigl\{F(\zeta)-J_{1-\eta}(\zeta)\bigr\} \notag\\
 &\qquad=
 -\inf_{\zeta\in\R^{2d}}
 \bigl\{J_{1-\eta}(\zeta)+c\Psi_\eta(q,\zeta)\bigr\}.
 \label{eq:terminal-laplace-principle}
\end{align}

We claim that, for every $c>0$,
\begin{equation}\label{eq:penalized-action-limit}
 \lim_{\eta\downarrow0}
 \inf_{\zeta\in\R^{2d}}
 \left\{
  J_{1-\eta}(\zeta)+c\Psi_\eta(q,\zeta)
 \right\}
 =J_1(q).
\end{equation}
Fix $c>0$ and write
$$
 a_\eta
 :=
 \inf_{\zeta\in\R^{2d}}
 \left\{
  J_{1-\eta}(\zeta)+c\Psi_\eta(q,\zeta)
 \right\}.
$$
For the upper bound, consider
$$
 \zeta_\eta=(q_1-\eta q_2,q_2).
$$
As
$$
 q_2-\zeta_{\eta,2}=0,
 \qquad
 q_1-\zeta_{\eta,1}-\eta q_2=0,
$$
we have  $\Psi_\eta(q,\zeta_\eta)=0$. Therefore,
$$
 \inf_{\zeta\in\R^{2d}}
 \left\{
  J_{1-\eta}(\zeta)+c\Psi_\eta(q,\zeta)
 \right\}
 \le J_{1-\eta}(\zeta_\eta).
$$
Since $(1-\eta,\zeta_\eta)\to(1,q)$ as $\eta\downarrow0$, the joint
continuity of $(r,\zeta)\mapsto J_r(\zeta)$ yields
$$
 \limsup_{\eta\downarrow0}a_\eta
 \le J_1(q).
$$
On the other hand, let $\eta_n\downarrow0$ be
arbitrary, and choose $\zeta_n\in\R^{2d}$ such that
$$
 J_{1-\eta_n}(\zeta_n)
 +c\Psi_{\eta_n}(q,\zeta_n)
 \le a_{\eta_n}+\frac1n.
$$
The preceding upper bound shows that $(a_{\eta_n})_{n\ge1}$ is bounded
above. Since $J_{1-\eta_n}\ge0$ and $c>0$, it follows that
$$
 \sup_{n\ge1}\Psi_{\eta_n}(q,\zeta_n)<\infty.
$$
Recalling the definition of $\Psi_\eta$, we obtain
$$
 |q_2-\zeta_{n,2}|=O(\eta_n^{1/2}),
 \qquad
 |q_1-\zeta_{n,1}-\eta_nq_2|
 =O(\eta_n^{3/2}).
$$
Consequently, we have $\zeta_n\to q$.
Combined with the joint continuity of
$J$, we obtain
$$
 J_{1-\eta_n}(\zeta_n)\longrightarrow J_1(q).
$$
Using the defining inequality for $\zeta_n$ and the non-negativity of
the penalty, we conclude that
$$
 a_{\eta_n}
 \ge
 J_{1-\eta_n}(\zeta_n)
 +c\Psi_{\eta_n}(q,\zeta_n)-\frac1n
 \ge
 J_{1-\eta_n}(\zeta_n)-\frac1n.
$$
Hence
$$
 \liminf_{n\to\infty}a_{\eta_n}\ge J_1(q).
$$
Since the sequence $\eta_n\downarrow0$ was arbitrary, this proves the
liminf inequality. Together with the upper bound, it establishes
\eqref{eq:penalized-action-limit}.

Since
$\varepsilon\log\bigl(C_i(\eta\varepsilon)^{-2d}\bigr)\to0$ for fixed
$\eta$, equations \eqref{eq:density-laplace-bounds} and
\eqref{eq:terminal-laplace-principle} imply
$$
\begin{aligned}
 -\inf_{\zeta\in\R^{2d}}
  \bigl\{J_{1-\eta}(\zeta)+c_0\Psi_\eta(q,\zeta)\bigr\}
 &\le \liminf_{\varepsilon\downarrow0}
  \varepsilon\log p_b(s,\bm z;s+\varepsilon,\bm y_\varepsilon)\\
 &\le \limsup_{\varepsilon\downarrow0}
  \varepsilon\log p_b(s,\bm z;s+\varepsilon,\bm y_\varepsilon)\\
 &\le -\inf_{\zeta\in\R^{2d}}
  \bigl\{J_{1-\eta}(\zeta)+c_1\Psi_\eta(q,\zeta)\bigr\}.
\end{aligned}
$$
Letting $\eta\downarrow0$ and invoking
\eqref{eq:penalized-action-limit} gives
$$
 \lim_{\varepsilon\downarrow0}
 \varepsilon\log p_b(s,\bm z;s+\varepsilon,\bm y_\varepsilon)=-J_1(q).
$$
This proves the theorem.
\end{proof}

\begin{proof}[Proof of Theorem~\ref{thm:fixed-endpoint}]
Throughout the proof, $s\ge0$ and $\bm z=(x,v)$ are fixed. For $t\in[0,1]$,
define
\begin{align*}
 \widehat Y_t^\varepsilon
 &:=
 \left(X_{s+\varepsilon t}-x-\varepsilon t v,\
       \varepsilon(V_{s+\varepsilon t}-v)\right),\\
 \widehat G_t^\varepsilon
 &:=
 \left(\sqrt2\,\varepsilon\int_0^t
       (W_{s+\varepsilon r}-W_s)\,\dd r,\
       \sqrt2\,\varepsilon(W_{s+\varepsilon t}-W_s)\right).
\end{align*}
Indeed, \eqref{eq:scaled-comparison} implies
$$
 \norm{\widehat Y^\varepsilon-\widehat G^\varepsilon}_\infty
 \le
 2\varepsilon\int_s^{s+\varepsilon}
 |b(r,X_r,V_r)|\,\dd r.
$$
Fix $a>0$. By Proposition~\ref{prop:khasminskii}, we have, for all sufficiently small $\varepsilon$,
$$
 \E_{s,\bm z}\exp\!\left(
 a\int_s^{s+\varepsilon}|b(r,X_r,V_r)|^2\,\dd r
 \right)\le2.
$$
Moreover, by the Cauchy--Schwarz inequality,
$$
 \left(
 \int_s^{s+\varepsilon}|b(r,X_r,V_r)|\,\dd r
 \right)^2
 \le
 \varepsilon\int_s^{s+\varepsilon}|b(r,X_r,V_r)|^2\,\dd r.
$$
Consequently,
\begin{align*}
 &\Pp_{s,\bm z}\!\left(
   \norm{\widehat Y^\varepsilon-\widehat G^\varepsilon}_\infty>\rho
 \right)\\
 &\qquad\le
 \Pp_{s,\bm z}\!\left(
   \int_s^{s+\varepsilon}|b(r,X_r,V_r)|^2\,\dd r
   >\frac{\rho^2}{4\varepsilon^3}
 \right)
 \le
 2\exp\!\left(-\frac{a\rho^2}{4\varepsilon^3}\right).
\end{align*}
Since $a>0$ is arbitrary, it follows that, for every $\rho>0$,
$$
 \limsup_{\varepsilon\downarrow0}\varepsilon^3
 \log\Pp_{s,\bm z}\!\left(
   \norm{\widehat Y^\varepsilon-\widehat G^\varepsilon}_\infty>\rho
 \right)
 =-\infty.
$$
Hence $\widehat Y^\varepsilon$ and $\widehat G^\varepsilon$ are
exponentially equivalent at speed $\varepsilon^{-3}$.

By the Brownian scaling property, we have
$$
 \widehat G^\varepsilon
 \stackrel{\mathrm{law}}{=}
 \varepsilon^{3/2}
 \left(\sqrt2\int_0^\cdot B_r\,\dd r,\sqrt2 B\right).
$$
Similarly to the proof of Theorem~\ref{thm:path-ldp}, Schilder's theorem and the
contraction principle therefore imply that
$\{\widehat G^\varepsilon\}_{\varepsilon>0}$ satisfies an LDP on
$C([0,1];\R^{2d})$ with speed $\varepsilon^{-3}$ and good rate function $I_1$ defined in \eqref{eq:time-r-action}. The exponential equivalence of $\widehat Y^\varepsilon$ and $\widehat G^\varepsilon$ implies that
$\{\widehat Y_r^\varepsilon\}_{\varepsilon>0}$ satisfies an LDP on $\R^{2d}$ with speed $\varepsilon^{-3}$ and rate function $J_r$ given by
\eqref{eq:endpoint-action-computed}.

Fix $\eta\in(0,1)$ and set
$$
 t_\eta:=s+(1-\eta)\varepsilon,\qquad
 \bm y_\varepsilon:=\theta_{s,s+\varepsilon}(\bm z)+q,\qquad
 \bar q:=(q_1,0).
$$
For $\zeta=(\zeta_1,\zeta_2)\in\R^{2d}$, define
$$
 \bm u_\varepsilon(\zeta)
 :=
 \theta_{s,t_\eta}(\bm z)
 +\left(\zeta_1,\frac{\zeta_2}{\varepsilon}\right).
$$
Thus, we have
\begin{align}\label{eq:transform}
 (X_{t_\eta},V_{t_\eta})
 =\bm u_\varepsilon\bigl(\widehat Y_{1-\eta}^\varepsilon\bigr).
\end{align}

With $\Psi_\eta$ as in \eqref{eq:terminal-cost}, a direct computation gives
\begin{align}
 &\rho_{\eta\varepsilon}\!\left(
 \bm y_\varepsilon-
 \theta_{t_\eta,s+\varepsilon}\bigl(\bm u_\varepsilon(\zeta)\bigr)
 \right)
 =
 \frac1{\varepsilon^3}
 \left\{
 \frac{|q_1-\zeta_1-\eta\varepsilon q_2|^2}{\eta^3}
 +\frac{|\varepsilon q_2-\zeta_2|^2}{\eta}
 \right\}\notag\\
 &\hspace{44mm}=
 \frac1{\varepsilon^3}
 \Psi_\eta\bigl(\bar q,\zeta-h_{\eta,\varepsilon}\bigr),
 \label{eq:fixed-endpoint-translation}
\end{align}
where
$$
 h_{\eta,\varepsilon}:=
 (-\eta\varepsilon q_2,\varepsilon q_2)\in\R^{2d}.
$$
 In the same form as
\eqref{eq:ck-density-transfer}, by the Chapman--Kolmogorov identity and \eqref{eq:transform}, we have
$$
 p_b(s,\bm z;s+\varepsilon,\bm y_\varepsilon)
 =
 \E_{s,\bm z}\!\left[
 p_b\!\left(
 t_\eta,
 \bm u_\varepsilon\bigl(\widehat Y_{1-\eta}^\varepsilon\bigr);
 s+\varepsilon,\bm y_\varepsilon
 \right)\right].
$$
By \eqref{eq:heat-kernel-bounds} and \eqref{eq:fixed-endpoint-translation}, we obtain that
\begin{equation}\label{eq:fixed-endpoint-density-bounds}
 C_0(\eta\varepsilon)^{-2d}\Lambda_\varepsilon(c_0)
 \le
 p_b(s,\bm z;s+\varepsilon,\bm y_\varepsilon)
 \le
 C_1(\eta\varepsilon)^{-2d}\Lambda_\varepsilon(c_1),
\end{equation}
where
$$
 \Lambda_\varepsilon(c)
 :=
 \E_{s,\bm z}\exp\!\left[
 -\frac{c}{\varepsilon^3}
 \Psi_\eta\!\left(
   \bar q,\widehat Y_{1-\eta}^\varepsilon-h_{\eta,\varepsilon}
 \right)\right].
$$
For each fixed $\eta\in(0,1)$, set
$$
 \widetilde Y_{\eta}^\varepsilon
 :=
 \widehat Y_{1-\eta}^\varepsilon-h_{\eta,\varepsilon}.
$$
Since
$$
 \bigl|
 \widetilde Y_{\eta}^\varepsilon
 -\widehat Y_{1-\eta}^\varepsilon
 \bigr|
 =|h_{\eta,\varepsilon}|
 \longrightarrow0,
$$
for every $\rho>0$ one has
$$
 \Pp_{s,\bm z}\!\left(
 \bigl|
 \widetilde Y_{\eta}^\varepsilon
 -\widehat Y_{1-\eta}^\varepsilon
 \bigr|>\rho
 \right)=0
$$
for all sufficiently small $\varepsilon$. Thus
$\{\widetilde Y_{\eta}^\varepsilon\}_{\varepsilon>0}$ and
$\{\widehat Y_{1-\eta}^\varepsilon\}_{\varepsilon>0}$ are exponentially
equivalent at speed $\varepsilon^{-3}$, and the former satisfies the
same LDP with rate function $J_{1-\eta}$. Since
$-c\Psi_\eta(\bar q,\cdot)$ is continuous and bounded above, Varadhan's
integral lemma yields, for every $c>0$,
\begin{equation}\label{eq:fixed-endpoint-laplace}
 \lim_{\varepsilon\downarrow0}
 \varepsilon^3\log\Lambda_\varepsilon(c)
 =
 -a_{\eta,c},
 \qquad
 a_{\eta,c}
 :=
 \inf_{\zeta\in\R^{2d}}
 \left\{
 J_{1-\eta}(\zeta)+c\Psi_\eta(\bar q,\zeta)
 \right\}.
\end{equation}

The variational limit required in
\eqref{eq:fixed-endpoint-laplace} has already been established in \eqref{eq:penalized-action-limit}. For every $c>0$, we have
\begin{equation}\label{eq:fixed-endpoint-variational-limit}
 \lim_{\eta\downarrow0}a_{\eta,c}
 =J_1(\bar q)=3|q_1|^2.
\end{equation}
Combining \eqref{eq:fixed-endpoint-density-bounds} with
\eqref{eq:fixed-endpoint-laplace} therefore yields
$$
 -a_{\eta,c_0}
 \le
 \liminf_{\varepsilon\downarrow0}\varepsilon^3
 \log p_b(s,\bm z;s+\varepsilon,\bm y_\varepsilon)
 \le
 \limsup_{\varepsilon\downarrow0}\varepsilon^3
 \log p_b(s,\bm z;s+\varepsilon,\bm y_\varepsilon)
 \le
 -a_{\eta,c_1}.
$$
Letting $\eta\downarrow0$ and invoking
\eqref{eq:fixed-endpoint-variational-limit}, we obtain that
$$
 \lim_{\varepsilon\downarrow0}\varepsilon^3\log
 p_b\!\left(s,\bm z;s+\varepsilon,
 \theta_{s,s+\varepsilon}(\bm z)+q\right)
 =-J_1(\bar q)=-3|q_1|^2,
$$
which proves \eqref{eq:fixed-endpoint}.

If $q_1=0$, then $D_\varepsilon(0,q_2)=(0,q_2)$. Hence
Theorem~\ref{thm:varadhan}, applied with the fixed vector $(0,q_2)$,
implies
$$
 \lim_{\varepsilon\downarrow0}\varepsilon\log
 p_b\!\left(s,\bm z;s+\varepsilon,
 \theta_{s,s+\varepsilon}(\bm z)+(0,q_2)\right)
 =-J_1(0,q_2)=-|q_2|^2,
$$
which proves \eqref{eq:pure-velocity-endpoint}.
\end{proof}

\vskip 0.3cm
\noindent{\large\bf Acknowledgements}

\begingroup
\emergencystretch=2em
This work is partially supported by the National Key R\&D Program of China (No. 2022YFA1006001) and the National Natural Science Foundation of China (Nos. 12131019, 12571158, 12371151). 
\par\endgroup

\end{document}